%% file: detection.tex
\documentclass[a4paper]{amsart}

\usepackage{amsmath,amstext,amssymb,mathrsfs,amscd,amsthm,indentfirst}
\usepackage{amsfonts}

\usepackage{booktabs}
\usepackage{verbatim}
\usepackage{enumerate}

\usepackage{graphicx}
\numberwithin{equation}{section}

\newcommand{\bQ}{\mathbb{Q}}
\newcommand{\bR}{\mathbb{R}}

\newcommand{\bZ}{\mathbb{Z}}

\renewcommand{\phi}{\varphi}

\newcommand{\MSO}{\mathbf{MSO}}
\newcommand{\MT}[2]{\mathbf{MT #1}(#2)}

\newcommand\lra{\longrightarrow}

\newcommand\Diff{\mathrm{Diff}}

\newcommand\Th{\mathrm{Th}}
\newcommand\colim{\operatorname*{colim}}
\newcommand\hocolim{\operatorname*{hocolim}}

\newcommand\Ker{\operatorname*{Ker}}

\newcommand{\ph}{\mathrm{ph}} 

\newcommand{\R}{\bR}

\newcommand{\Hom}{\mathrm{Hom}}

\renewcommand{\epsilon}{\varepsilon}

\newcommand{\Gr}{\mathrm{Gr}}

\newcommand{\CircNum}[1]{\ooalign{\hfil\raise .00ex\hbox{\scriptsize #1}\hfil\crcr\mathhexbox20D}}

\mathchardef\ordinarycolon\mathcode`\:
\mathcode`\:=\string"8000
\begingroup \catcode`\:=\active
  \gdef:{\mathrel{\mathop\ordinarycolon}}
\endgroup

\usepackage[all]{xy}
\SelectTips{eu}{10}
\CompileMatrices

\usepackage{amsthm}
\theoremstyle{plain}

\newtheorem{theorem}{Theorem}[section]

\newtheorem{proposition}[theorem]{Proposition}
\newtheorem{lemma}[theorem]{Lemma}

\newtheorem{corollary}[theorem]{Corollary}

\theoremstyle{definition}
\newtheorem{definition}[theorem]{Definition}

\newtheorem{claim}[theorem]{Claim}

\theoremstyle{remark}

\newtheorem{remark}[theorem]{Remark}

\newtheorem*{remark*}{Remark}

\title[Characteristic classes of manifold bundles]{Detecting and realising characteristic classes of manifold bundles}

\author{S{\o}ren Galatius}
\thanks{S. Galatius was partially supported by NSF grant DMS-1105058, and both authors were supported by ERC Advanced Grant No.\ 228082, and the Danish National Research Foundation (DNRF) through the Centre for Symmetry and Deformation.}
\email{galatius@stanford.edu}
\address{Department of Mathematics\\
	Stanford University\\
	Stanford CA, 94305}

\author{Oscar Randal-Williams}
\email{o.randal-williams@dpmms.cam.ac.uk}
\address{DPMMS\\
Wilberforce Road\\
Cambridge CB3 0WB\\
UK}


\begin{document}
\begin{abstract}
  We apply our earlier work on the higher-dimensional analogue of the
  Mumford conjecture to two questions.  Inspired by work of Ebert
  (\cite{Ebert09}, \cite{Ebert09Independence}) we prove non-triviality
  of certain characteristic classes of bundles of smooth closed
  manifolds.  Inspired by work of Church--Farb--Thibault and
  Church--Crossley--Giansiracusa (\cite{CFT}, \cite{CCG}) we
  investigate the dependence of characteristic classes of bundles on
  characteristic numbers of its fibre, total space and base space.
\end{abstract}
\maketitle

\input{chap1}
\input{chap2}
\input{chap3}

\bibliographystyle{amsalpha}
\bibliography{biblio}

\end{document}

%% file: chap1.tex
\section{Introduction and statement of results}
\label{sec:intr-stat-results}

A smooth bundle of closed oriented $d$-manifolds $\pi : E \to B$ is a
proper submersion with $d$-dimensional fibres, together with an
orientation of the \emph{vertical tangent bundle} $T_\pi E =
\Ker(D\pi)$.  A characteristic class of such bundles associates to
this data an element in $H^*(B)$ which is natural in the bundle.  A
useful way to define such classes goes via the \emph{parametrised
  Pontryagin--Thom construction}, which to each $\pi: E \to B$
associates a map
\begin{equation*}
  \overline{\alpha}_\pi: B\lra \Omega^\infty_0 \MT{SO}{d},
\end{equation*}
well defined up to homotopy and natural with respect to pull back of
bundles.  We recall the definition of the space $\Omega^\infty \MT{SO}{d}$
and the construction of $\overline{\alpha}_\pi$ below.  Each
cohomology class $c \in H^*(\Omega^\infty_0\MT{SO}{d})$ gives a
characteristic class $\overline{\alpha}_\pi^*(c) \in H^*(B)$, but it
is far from clear whether these characteristic classes are
non-trivial.  This is the \emph{detection question}, i.e.\ if $c \neq
0$, does there exist a bundle $\pi: E \to B$ with
$\overline{\alpha}_\pi^*(c) \neq 0 \in H^*(B)$?  When cohomology is
taken with rational coefficients, Ebert
(\cite{Ebert09Independence}, \cite{Ebert09}) has obtained a complete
answer for all $d$.  (He proves that the answer is ``yes'' when $d$ is
even and ``no'' when $d$ is odd.)  In this paper we consider the
detection question with arbitrary coefficients.  We shall give an
affirmative answer when $d$ is an even number greater than 4, and also
obtain more control over the detecting bundles.  The following is our
main result in this direction.
\begin{theorem}\label{thm:main}
  Let $2n \neq 4$ and $f \in \Omega_{2n}^{SO}$ be a bordism class.
  Let $k$ be an abelian group and $c \in H^*(\Omega^\infty_0
  \MT{SO}{2n};k)$ be a non-zero class. Then there exists a smooth
  bundle of closed oriented $2n$-manifolds $\pi: E \to B$ such that
  \begin{enumerate}[(i)]
  \item $\overline{\alpha}_\pi^*(c) \neq 0 \in H^*(B;k)$,
  \item\label{item:2} the fibres of $\pi$ lie in the bordism class $f$,
  \item\label{item:1} there is a manifold $M$ with boundary, such that
    $(M, \partial M)$ is $(n-1)$-connected and the Gauss map $M \to
    BSO$ is $n$-connected, with a fibrewise embedding of a trivial
    subbundle $B \times M \subset E$ such that $(E, B \times M)$ is
    $(n-1)$-connected.
  \end{enumerate}
\end{theorem}

For $k = \bQ$, Theorem~\ref{thm:main} reproduces Ebert's result, but
with an entirely different proof.  Ebert's argument used the known
structure of $H^*(\Omega^\infty_0 \MT{SO}{2n};\bQ)$ as a Hopf algebra
to reduce his detection result to certain explicit characteristic
classes $\kappa_c \in H^p(\Omega^\infty_0 \MT{SO}{2n};\bQ)$, the
``generalised Miller--Morita--Mumford classes'', associated to a
monomial $c \in H^{p + 2n}(BSO(2n))$ in the Euler class and Pontryagin
classes (we will describe these $\kappa_c$ below).  He then uses an
explicit construction of smooth bundles to detect these classes.  We
will deduce Theorem~\ref{thm:main} as an application of the
higher-dimensional version of the Madsen--Weiss theorem from our
previous paper \cite{GR-W2}.

By Theorem~\ref{thm:main}, characteristic classes can be detected
using bundles whose fibre lies in a prescribed bordism class.  We also
investigate the question of detecting classes when the bordism class
of the total space and base space is prescribed.  Before describing
our results in this direction, we recall some definitions.

\subsection{The parametrised Pontryagin--Thom construction}
\label{sec:param-pontry-thom}

We recall the definition of the space $\Omega^\infty_0 \MT{SO}{d}$,
and the map $\overline{\alpha}_\pi: B \to \Omega^\infty_0 \MT{SO}{d}$
associated to a smooth bundle of closed oriented $d$-manifolds $\pi: E
\to B$.

Let $\gamma_N^\perp$ denote the $(N-d)$-dimensional bundle over the
oriented Grassmannian $\Gr_d^+(\R^N)$ and let
$\mathrm{Th}(\gamma_N^\perp)$ denote its Thom space.  There is a
natural map $S^1 \wedge \Th(\gamma_N^\perp) \to
\Th(\gamma_{N+1}^\perp)$, and we define the infinite loop space
\begin{align*}
  \Omega^\infty \MT{SO}{d} = \colim_{N \to \infty} \Omega^N
  \mathrm{Th}(\gamma_N^\perp).
\end{align*}
Given a smooth bundle of closed oriented $d$-manifolds $\pi: E \to B$,
we may pick an embedding $j: E \to B \times \R^N$ over $B$, and extend
to an open embedding of the fibrewise normal bundle $\nu E \to B
\times \R^N$.  Then the Pontryagin--Thom collapse construction gives a
based map $B_+ \wedge S^N \to \Th(\nu)$.  The embedding $j$ also
induces a map $\tau: E \to \Gr_d^+(\R^N)$ with $\tau^*(\gamma^\perp_N) =
\nu$, and we have a composition $B_+ \wedge S^N \to \Th(\nu) \to
\Th(\gamma^\perp_N)$, whose adjoint gives a map
\begin{align*}
  \alpha_\pi: B \to \Omega^N \Th(\gamma_N^\perp) \subset \Omega^\infty
  \MT{SO}{d}.
\end{align*}
For large $N$, the embedding $j: E \to B \times \R^N$ is unique up to
isotopy, so the map $\alpha_\pi: B \to \Omega^\infty \MT{SO}{d}$ has
well defined homotopy class, depending only on the bundle $\pi: E \to
B$.  All path components of $\Omega^\infty \MT{SO}{d}$ are homotopy
equivalent, and we let $\overline{\alpha}_\pi: B \to
\Omega^\infty_0 \MT{SO}{d}$ denote the translation to the
basepoint component.

Let us now recall the definition of the \emph{generalised
  Miller--Morita--Mumford classes}.  These are universal classes
\begin{equation*}
  \kappa_c \in H^p(\Omega^\infty \MT{SO}{d})
\end{equation*}
associated to a class $c \in H^{p+d}(BSO(d))$ by applying the composition
\begin{equation*}
  H^{p+d}(BSO(d)) \overset{\text{Thom}}\lra H^p(\MT{SO}{d}) \overset{\sigma}\lra H^p(\Omega^\infty \MT{SO}{d}),
\end{equation*}
where the first map is the Thom isomorphism, and the second map is the
cohomology suspension. Given a bundle $\pi: E \to B$ with
corresponding Pontryagin--Thom map $\overline{\alpha}_\pi$, we obtain
the equation
$$\overline{\alpha}_\pi^*(\kappa_c) = \pi_!(c(T_\pi E)) \in H^{p}(B),$$
where $c(T_\pi E) \in H^{p+d}(E)$ denotes the characteristic class of
oriented vector bundles $c$ applied to te bundle $T_\pi E$, and
$\pi_!: H^{p+d}(E) \to H^p(B)$ denotes the Gysin map (fibrewise
integration).

Finally let us recall the rational cohomology of
$\Omega^\infty_0\MT{SO}{d}$, in the case where $d$ is even, say $d =
2n$.  To this end, let $\mathcal{B} \subset H^*(BSO(2n);\bQ)$ be the
set of monomials in the Euler class and the Pontryagin classes whose total
degree is greater than $2n$.  Then the natural map induces an
isomorphism
\begin{equation*}
  \bQ[\kappa_c \,\vert\, c \in \mathcal{B}] \overset{\cong}{\lra}
  H^*(\Omega^\infty_0 \MT{SO}{2n};\bQ).
\end{equation*}

\subsection{Bundles with prescribed characteristic numbers}
\label{sec:bundl-with-prescr}

If $\pi:E \to B$ is a smooth bundle of closed oriented
$2n$-dimensional manifolds and $B$ is also closed, of dimension $p$
say, then for each $\kappa \in H^p(\Omega^\infty_0\MT{SO}{2n})$ we get
a number
\begin{equation*}
  K(\kappa) = \int_B \overline{\alpha}_\pi^*(\kappa) \in \bZ
\end{equation*}
which we call the \emph{characteristic number} of the bundle
associated to the class $\kappa$.  The process $\kappa \mapsto K(\kappa)$ is
additive and gives a homomorphism
\begin{equation*}
  K: H^p(\Omega^\infty_0\MT{SO}{2n}) \lra \bZ.
\end{equation*}
For $d = 2n$ and each sequence $C = (c_1, \dots, c_r)$ of elements in
$\mathcal{B}$ we get a class $\kappa_C = \kappa_{c_1} \dots
\kappa_{c_r} \in H^*(\Omega^\infty_0 \MT{SO}{2n})$, and the
homomorphism $K$ is uniquely determined by the characteristic numbers
\begin{equation*}
  K_C = K(\kappa_C),
\end{equation*}
as $C = (c_1, \dots, c_r)$ runs through sequences of elements of $B$
with $p = \sum(|c_i| - 2n)$.

Theorem~\ref{thm:main} can be combined with Thom's theorem that any
non-zero cohomology class in $H^p(B;\bQ)$ is detected on some smooth
$p$-manifold $M \to B$ to deduce that for any non-trivial $C = (c_1,
\dots, c_r)$ with $p = \sum(|c_i| - 2n)$ there exists a smooth bundle
$\pi: E \to B$ where $B$ is a closed $p$-manifold, such that $K_C \neq
0$ , and we may even prescribe the bordism class of its fibre.  It is
a natural question to ask to what extent it is possible to find
bundles $\pi: E \to B$ with prescribed values of all characteristic
numbers, perhaps also prescribing the bordism classes $[F]$, $[E]$ and
$[B]$ of its fibres, total space, and base space.  In this form the
question is extremely difficult, and involves the integral homology
$H^*(\Omega^\infty_0 \MT{SO}{d})$ which is unknown for $d > 2$ (and
very complicated for $d=2$, cf.\ \cite{galatius-2004}).  However, if
we only ask for realising values of $[E]$, $[B]$ and the numbers $K_C$
up to multiplication by the same positive integer, we can give a
complete answer.  In other words, for a bundle $\pi: E \to B$ with
$2n$-dimensional fibres which are all in the same bordism class $[F]$
and $p$-dimensional base, the bordism classes $[E]$ and $[B]$ and the
characteristic numbers determine a point in the projectivisation of
the rational vector space
\begin{equation*}
  \big(\Omega^{SO}_{2n+p} \oplus \Omega^{SO}_p \oplus
  \Hom(H^p(\Omega^\infty_0\MT{SO}{2n}),\bZ) \big) \otimes \bQ.
\end{equation*}
We prove that the set of points which can be realised by bundles is a
rational subvariety cut out by the linear equations~\eqref{eq:3}
described below.  (The subvariety depends on the bordism class $[F]$
of the fibres of $\pi$.)

Our interest in the question of (in)dependence of characteristic
number of bundles from bordism classes of fibre, total space and base
space came from work of Church, Farb and Thibault: they proved in
\cite{CFT} that certain characteristic numbers of surface bundles
depend only on the characteristic numbers of its total space.  That
work was generalised to higher dimensions by Church, Crossley and
Giansiracusa: in \cite{CCG} they completely classified characteristic
numbers of bundles of oriented $d$-manifolds which depend only on the
oriented bordism class of its total space (i.e.\ not on the base $B$ or
the map $\pi$).  For $d = 2n$, their equations expressing a
characteristic number of a bundle as a characteristic number of its
total space will be contained in our equations.  We shall answer the
related question of characteristic numbers of bundles which depend
only on the bordism classes of the total space $E$ \emph{and} the base
$B$ (but still not the map $\pi: E \to B$; in fact, the
Church--Crossley--Giansiracusa equations are precisely those of our
equations which do not involve the base and the fibre).  We also
address the corresponding realisation question: Any set of solutions
to our equations is, up to scaling by positive integers, realised by a
bundle.

Suggested by the notion of ``near-primitive elements'' of \cite{CCG}
we make the following definition, which we will use to formulate our
equations.
\begin{definition}
  Let $d \geq 2$ and $\rho : H^*(BSO;\bQ) \to H^*(BSO(d);\bQ)$ be the
  restriction map. We say a class $x \in H^{p+d}(BSO;\bQ)$ is
  \emph{almost primitive of order $d$} if the coproduct
  $\Delta(x)$ can be written
  \begin{equation}\label{eq:2}
    \Delta(x) = 1 \otimes x + \sum x^p_j \otimes x^{d}_j + \sum a_i
    \otimes b_i
  \end{equation}
  where $x^r_j$ has degree $r$ and each $b_i$ either has degree $< d$
  or lies in $\Ker(\rho)$.  In other words, $x$ is almost primitive of
  order $d$ when it is sent to $1 \otimes \rho(x)$ under the
  composition
  \begin{equation*}
    H^*(BSO) \overset{\Delta}\lra H^*(BSO) \otimes H^*(BSO)
    \xrightarrow{\mathrm{Id} \otimes \text{proj}\circ \rho} H^*(BSO)
    \otimes H^{* \geq d+1}(BSO(d)).
  \end{equation*}
  We write $AP^*(d) \subset
  H^*(BSO;\bQ)$ for the vector subspace of such elements.
\end{definition}

Explicitly, the space of almost primitive elements is described by the
following proposition.  Let $\ph_i \in H^{4i}(BSO;\bQ)$ denote the
$i$th Pontryagin character class (which is primitive for the coproduct
$\Delta$), so $H^*(BSO;\bQ) \cong \bQ[\ph_1, \ph_2, \ldots]$.

\begin{proposition}\label{prop:Class}
  For $d \geq 2$, the vector subspace $AP^*(d) \subset \bQ[\ph_1,
  \ph_2, \ldots]$ is spanned by those monomials in the $\ph_i$'s
  having the property that every proper factor has degree $\leq d$.
  (In particular, the $\ph_i$
  themselves are almost primitive of any degree.)
\end{proposition}

Using the notion of almost-primitive elements, we now describe the
linear equations relating $[E]$, $[B]$ and the characteristic numbers
of a bundle $\pi: E \to B$.  Suppose $B$ has dimension $p$ and $E$ has
dimension $p + 2n$, let $x \in AP^{p+2n}(2n) \subset
H^{p+2n}(BSO;\bQ)$, and $\rho(x) \in H^*(BSO(2n);\bQ)$ denote its
restriction.  Using~\eqref{eq:2}, the bundle isomorphism $TE \cong
\pi^*(TB) \oplus T_\pi E$ gives
\begin{equation*}
  x(TE) = \rho(x)(T_\pi E) + \sum_j \pi^*(x^p_j(TB)) \cdot
  x^{2n}_j(T_\pi E) + \sum_i \pi^*(a_i(TB)) \cdot b_i(T_\pi E)
\end{equation*}
where $b_i(T_\pi E)$ is a characteristic class of $T_\pi E$ of degree
$< 2n$.  Applying $\pi_!$ we obtain the equation
\begin{equation*}
  \pi_!(x(TE)) = \kappa_{\rho(x)}(\pi) + \sum_j x^p_j(TB) \cdot \int_F
  x^{2n}_j(TF) \in H^{p}(B;\bQ),
\end{equation*}
and so by integrating over $B$,
\begin{equation*}
  \int_E x(TE) = \int_B \kappa_{\rho(x)}(\pi) + \sum_j \left(\int_B
    x_j^p(TB)\right)\cdot\left( \int_F x_j^{2n}(TF)\right). 
\end{equation*}
The first term on the right hand side is a characteristic number of
the bundle $\pi$.  If we write characteristic numbers of closed
manifolds as $\langle x, [E]\rangle = \int_E x(TE)$ and similarly for
$B$ and $F$, the equation can be written as
\begin{equation}
  \label{eq:3}
  \langle x, [E] \rangle = K(\kappa_{\rho(x)}) + \sum_j \langle x_j^p, [B]\rangle
  \cdot \langle x_j^{2n}, [F] \rangle.
\end{equation}
For a fixed fibre $F$, we view~\eqref{eq:3} as a set of linear
equations between the bordism classes $[E]$ and $[B]$ and the
characteristic numbers $K(\kappa_{\rho(x)})$, one equation for each
element $x$ in a basis for the vector space $AP^{p+2n}(2n)$.  Our
second result is that any formal solution to these relations may be
realised by a fibre bundle, up to multiplying by a positive integer.
\begin{theorem}\label{thm:main2}
  Let $2n \in \{2, 6, 8, 10, \ldots\}$ and $p > 0$. Fix
  \begin{enumerate}[(i)]
  \item a class $f \in \Omega^{SO}_{2n}$,
  \item a class $e \in \Omega^{SO}_{2n+p}$,
  \item a class $b \in \Omega_p^{SO}$,
  \item\label{it:4} a homomorphism $K: H^{p}(\Omega^\infty_0
    \MT{SO}{2n}) \to \bZ$.
  \end{enumerate}
  Suppose that for each $x \in AP^{2n+p}(2n)$, with coproduct as
  in~\eqref{eq:2}, we have
  \begin{equation}
    \label{eq:4}
    \langle x, e \rangle = K(\kappa_{\rho(x)}) + \sum_j \langle x_j^p, b\rangle
    \cdot \langle x_j^{2n}, f \rangle.
  \end{equation}
  Then there exists an integer $N > 0$ and a fibre bundle $\pi: E \to
  B$ satisfying condition (\ref{item:1}) in Theorem~\ref{thm:main},
  with $[E] = Ne$, $[B]=Nb$, and $\int_B \kappa_C(\pi) = NK(\kappa_C)$
  for all sequences $C = (c_1, \dots, c_r)$ of elements of
  $\mathcal{B}$ with $p = \sum (|c_i| - 2n)$, and with fibres in the
  bordism class $f$.
\end{theorem}

If we do not care about the characteristic numbers $K(\kappa_C)$, we
can pick them so $K(\kappa_{\rho(x)})$ satisfies the equations.  We
get the following corollary.
\begin{corollary}\label{cor:fibrre}
  Let $2n \in \{2, 6, 8, 10, \ldots\}$ and $p > 0$. Fix classes $f, e,
  b \in \Omega_*^{SO}$ with degrees $2n$, $2n+p$, and $p$
  respectively.  Then there exists an integer $N > 0$ and a fibre
  bundle $\pi: E \to B$, satisfying condition~(\ref{item:1}) of
  Theorem~\ref{thm:main}, such that $[E] = Ne$ and $[B]=Nb$, and such
  that the fibres of $\pi$ are in the bordism class $f$.\qed
\end{corollary}

\begin{remark}
  Condition~(\ref{item:1}) in the conclusion of Theorem~\ref{thm:main}
  implies that the fibres of $\pi$ satisfy that the Gauss map to $BSO$
  is $n$-connected.  By work of Kreck, such manifolds have very strong
  rigidity properties.  Indeed, by \cite[Theorem C]{Kreck}, two such
  are oriented diffeomorphic if and only if they are oriented bordant
  and have the same Euler characteristic, provided the Euler
  characteristic is sufficiently large.

  Therefore, the statements in Theorem~\ref{thm:main},
  Theorem~\ref{thm:main2} and Corollary~\ref{cor:fibrre} that the
  fibres of the bundle $\pi: E \to B$ are in the prescribed bordism
  class $f \in \Omega^{SO}_{2n}$ could as well have been stated as
  follows: given a closed $2n$-dimensional manifold $F$ whose Gauss
  map $F \to BSO$ is $n$-connected, there is a $g \gg 0$ such that all
  fibres of $\pi$ are diffeomorphic to $F \# g(S^n \times S^n)$.
\end{remark}


%% file: chap2.tex
\section{Proof of Theorem \ref{thm:main}}

Let the data $f \in \Omega_{2n}^{SO}$ and $0 \neq c \in
H^*(\Omega^\infty_0 \MT{SO}{2n};k)$ be given.

\begin{lemma}\label{lem:Const}
  There is a manifold $F$ in the bordism class $f$ having the property
  that the (normal) Gauss map $\nu_F:F \to BSO$ is $n$-connected.
\end{lemma}
\begin{proof}
  This may be found in e.g.\ \cite[Theorem 7]{KreckUniqueDiff}. We
  only require the result for even-dimensional manifolds, which is
  much easier, so we give it here.

  For each $x \in \pi_k(BSO)$ with $k \leq n$, we can find a lift to
  $\hat{x} \in \pi_k(BSO(2n-k+1))$, i.e.\ a vector bundle
  $V_x^{2n-k+1} \to S^k$ representing this class. The sphere bundle of
  $V$, $S(V_x)$, is a $2n$-manifold and by construction has the
  property that its normal Gauss map $\nu_{S(V_x)} : S(V_x) \to BSO$
  hits the class $x$ on $\pi_k$.

  If we choose any $F'' \in f$ we may take connected-sum with $S(V_x)$
  for generators $x$ of $\pi_{* \leq n}(BSO)$ to obtain a manifold
  $F'$ such that $\nu_{F'}$ is surjective on homotopy groups in
  degrees $* \leq n$. Furthermore, $[F'']=f$ as the manifolds $S(V_x)$
  are nullbordant (they bound the associated disc-bundle). Now, we
  perform surgery on $F'$ along spheres generating the kernels of
  $\pi_k(F') \to \pi_k(BSO)$ for $k \leq n-1$. The manifold we end up
  with, $F$, has $\nu_F$ $n$-connected.
\end{proof}

\begin{proof}[Proof of Theorem \ref{thm:main}]
  Let $F$ be a manifold satisfying the conclusion of
  Lemma~\ref{lem:Const}, and let $h : F \to [0,2n]$ be a self-indexing
  Morse function.  Let
  $P = h^{-1}(n - 1/2)$ 
  and define
  \begin{align*}
    W &= h^{-1}([0, n-\frac12]),\\
    A &= h^{-1}([n-\frac12, n+\frac12]),\\
    M &= h^{-1}([n-\frac12, 2n]).
  \end{align*}
  Then $A$ is a bordism from 
  $P$ to $h^{-1}(n + 1/2)$,
  and we let $\overline{A}$ denote the bordism in the other direction
  with opposite orientation.  Then define bordisms from $P$ to itself
  by
  \begin{align*}
    K_0 &= A \cup_{h^{-1}(n+1/2)} \overline{A} \\
    K_1 &= ([0,1] \times P)\# (S^n \times S^n)\\
    K &= K_0 \cup_P K_1.
  \end{align*}
  It is easy to see that $K_0$ and $K_1$ commute, in the sense that
  there is a diffeomorphism $K_0 \cup_P K_1 \cong K_1 \cup_P K_0$,
  relative to the boundary.

  By construction, this data $(W, K)$ satisfies the assumptions
  explained in \cite[Remark 1.11]{GR-W2}\footnote{The statement of
    \cite[Theorem 1.8]{GR-W2} says $2n > 4$, but it also holds for
    $2n=2$ by the Madsen--Weiss theorem \cite{MW}, and for $2n=0$ by
    the Barratt--Priddy theorem \cite{BP}.}, and furthermore the $n$th
  stage of the Moore--Postnikov factorisation of the tangential Gauss
  map $K \to BO(2n)$ is $BSO(2n) \to BO(2n)$.  Thus the
  Pontryagin--Thom construction gives a map
  \begin{equation}\label{eq:5}
    \hocolim_{g \to \infty}  B\Diff(W \cup gK, \partial) \lra
    \Omega^\infty_0 \MT{SO}{2n}
  \end{equation}
  and it follows from \cite[Theorem 1.8]{GR-W2} that this map induces
  an isomorphism in integral homology and hence in any generalised
  homology theory.

  Now let $c\neq 0 \in H^p(\Omega^\infty_0 \MT{SO}{2n};k)$.  Since
  $H_*(\Omega^\infty_0 \MT{SO}{2n})$ is finitely generated in each
  degree, $H^*(\Omega^\infty_0 \MT{SO}{2n};k)$ is the direct limit of
  cohomology with coefficients in finitely generated subgroups of $k$,
  so $c$ is in the image of cohomology with coefficients in some
  finitely generated subgroup $k' \subset k$.  Then there exists a
  class $x \in H_p(\Omega^\infty_0 \MT{SO}{2n};\Hom(k', \bQ/\bZ))$
  with $\langle x, c\rangle \neq 0 \in \bQ/\bZ$.  We proved above that
  $x$ is in the image from homology of some
  \begin{equation*}
    B\Diff(W \cup_P gK, \partial) \simeq B\Diff(W \cup_P gK \cup_P M,
    M),
  \end{equation*}
  and since any homology class is supported on a finite subcomplex of
  a CW approximation, there exists a manifold $B$ and a map $f: B \to
  B\Diff(W \cup_P gK \cup_P M, M)$ such that $x$ is in the image of
  $f_*$.  It follows that
  \begin{equation*}
    f^*(c) \neq 0 \in H^p(B; k)
  \end{equation*}
  and hence that the map $f$ classifies a bundle $\pi: E \to B$ with
  all the required properties.
\end{proof}


%% file: chap3.tex
\section{Proof of Theorem \ref{thm:main2}}

The strategy of the proof will be to first solve a bordism version of
the problem of Theorem \ref{thm:main2}, and then appeal to the results
of \cite{GR-W2} to upgrade this bordism solution to a fibre bundle
solution. The bordism version does not have the restriction $d = 2n
\neq 4$, it holds in any dimension $d \geq 2$.

Recall that the classes $\kappa_c$ are defined universally in
$H^{\vert c \vert-d}(\Omega^\infty \MT{SO}{d})$, and that the map
\begin{equation*}
  \bQ[\kappa_c \, \vert \, c \in \mathcal{C}] \lra H^*(\Omega^\infty
  \MT{SO}{d};\bQ)^{\pi_0} \cong H^*(\Omega^\infty_0 \MT{SO}{d};\bQ)
\end{equation*}
to the subring of $\pi_0(\Omega^\infty \MT{SO}{d})$-invariant classes
is an isomorphism, where $\mathcal{C}$ is the set of monomials of
degree $>d$ in $H^*(BSO(d);\bQ)$ in the variables $p_1, ...,
p_{\lfloor d/2 \rfloor}$ if $d$ is odd or $p_1, ..., p_{\lfloor d/2
  \rfloor-1}, e$ if $d$ is even.

By Pontryagin--Thom theory, the infinite loop space $\Omega^\infty
\MT{SO}{d}$ classifies ``formal oriented $d$-dimensional fibre
bundles''. That is, for any smooth manifold $B$ the Pontryagin--Thom
construction gives a natural bijection
\begin{equation*}
  [B, \Omega^\infty \MT{SO}{d}] \longleftrightarrow \left\{
    \begin{array}{l}
      \pi: E \to B \,\, \text{smooth proper map,} \\
      V \to E \,\,\text{$d$-dimensional oriented vector bundle,}\\
      \varphi: TE \simeq_s V \oplus \pi^*(TB) \,\,\text{stable isomorphism.}
    \end{array}
  \right\}
\end{equation*}
between homotopy classes of maps and cobordism classes over $B$ of
formal fibre bundles.  The Miller--Morita--Mumford classes can be defined in this theory:
to $c \in H^{p+d}(BSO(d))$ and $[\pi,E,B,V,\phi] \in [B, \Omega^\infty
\MT{SO}{d}]$ we let
\begin{equation*}
  \kappa_c = \pi_!(c(V)) \in H^{p}(B),
\end{equation*}
where the Gysin map $\pi_!$ is defined using Poincar{\'e} duality in $E$ and $B$.
Under the bijection, associating the map $\alpha_\pi: B \to
\Omega^\infty\MT{SO}{d}$ to a bundle $\pi: E \to B$ with oriented
$d$-dimensional fibres corresponds to forgetting that $\pi$ is a
bundle, remembering only the stable isomorphism $TE \cong_s \pi^* TB
\oplus T_\pi E$ induced by the differential of $\pi$.  For the
universal bundle with fibre $F$ we get the map
\begin{equation*}
  B\Diff(F) \lra \Omega^\infty \MT{SO}{d}.
\end{equation*}

Writing $\Omega_*^{SO}(-)$ for oriented bordism theory, there is a
natural bijection
\begin{equation*}
  \Omega_p^{SO}(\Omega^\infty \MT{SO}{d}) \longleftrightarrow \left\{
    \begin{array}{l}
      \pi: E^{p+d} \to B^k \,\, \text{smooth proper map,} \\
      V^{d} \to E \,\,\text{oriented vector bundle,}\\
      \varphi: TE \simeq_s V \oplus \pi^*(TB) \,\,\text{stable isomorphism.}
    \end{array}
  \right\}
\end{equation*}
to the set of cobordism classes of such data, where $B$ may also
changed by a cobordism.  The data relevant for Theorem~\ref{thm:main2}
can all be extracted from this group: the bordism classes $[F]$, $[E]$,
and $[B]$, and the characteristic numbers $K_C$.  Firstly, we can
define characteristic numbers $K_C$
\begin{equation*}
  K_C(\pi, E, B, V) = \int_B \kappa_{c_1}(\pi, V) \cdots
  \kappa_{c_n}(\pi, V) \in \bZ.
\end{equation*}
It is easy to check that these characteristic numbers are invariants
of the cobordism class.  There is an obvious map
$\Omega^{SO}_p(\Omega^\infty\MT{SO}{d}) \to \Omega^{SO}_p(\ast)$ which
sends $[\pi,E,B,V,\phi]$ to $[B]$.  Similarly, sending a bordism class
$[\pi, E, B, V, \phi]$ to the bordism class $[E] \in
\Omega^{SO}_{p+d}$ gives a well defined homomorphism
\begin{equation*}
  \Omega^{SO}_p(\Omega^\infty\MT{SO}{d}) \to \Omega^{SO}_{p+d}.
\end{equation*}
We point out that this homomorphism is \emph{not} invariant under
translation to different path components, and in this section it often
is better \emph{not} to translate back to the path component of the
basepoint.
In fact, the bordism class of the fibre $[F]$ corresponds to the group
of path components of $\Omega^\infty\MT{SO}{d}$.  There is a
stabilisation map $\MT{SO}{d} \to \Sigma^{-d} \MSO$ which induces a
surjection
\begin{equation*}
  \pi_0(\Omega^\infty \MT{SO}{d}) \lra \Omega_d^{SO}.
\end{equation*}
If $f \in \Omega_d^{SO}$ is a bordism class we write
$\Omega_{(f)}^\infty \MT{SO}{d}$ for the collection of path components
which go to $f$ under this map.  If $\pi: E \to B$ is a bundle with
fibres in the bordism class $f$, then the corresponding map
$\alpha_\pi: B \to \Omega^\infty\MT{SO}{d}$ has image in
$\Omega_{(f)}^\infty \MT{SO}{d}$.

The cobordism version of Theorem \ref{thm:main2} is the problem of
finding a class
\begin{equation*}
  [\pi, E, B, V, \varphi] \in \Omega_{d+p}^{SO}(\Omega^\infty_{(f)} \MT{SO}{d})
\end{equation*}
with $[E] = e \in \Omega_{d+p}^{SO}$, $[B]=b \in \Omega_p^{SO}$ and
which maps to a given functional $K \in \Hom(H^p(\Omega^\infty_0
\MT{SO}{d}), \bZ) \subset \Hom(H^p(\Omega^\infty_0 \MT{SO}{d}), \bQ)$
under
\begin{align*}
  {\Omega_p^{SO}(\Omega^\infty \MT{SO}{d})} \lra {{H}_p(\Omega^\infty
    \MT{SO}{d};\bQ)} & \lra {H_p(\Omega^\infty_0
    \MT{SO}{d};\bQ)}\\
  & = \Hom(H^p(\Omega^\infty_0 \MT{SO}{d}), \bQ),
\end{align*}
where the last map is induced by translating to the basepoint
component.

\begin{proposition}\label{prop:CobMain2}
  Let $d \geq 2$ and $p >0$ and fix
  \begin{enumerate}[(i)]
  \item a class $f \in \Omega^{SO}_{d}$,
  \item a class $e \in \Omega^{SO}_{d+p}$,
  \item a class $b \in \Omega^{SO}_{p}$,
  \item a homomorphism $K: H^{p}(\Omega^\infty_0 \MT{SO}{d}) \to \bZ$.
  \end{enumerate}
  Suppose that these data satisfy equation~\eqref{eq:4} for each $x
  \in AP^{d+p}(d)$.  Then there exists an integer $N > 0$ and a class
  \begin{equation*}
    [\pi, E, B, V, \varphi] \in \Omega_p^{SO}(\Omega_{(f)}^\infty
    \MT{SO}{d}_+)
  \end{equation*}
  such that $[E] = Ne$, $[B]=Nb$ and $\int_B \kappa_C(\pi, V) = NK_C$
  for all sequence $C = (c_1, \dots, c_r)$ of elements of
  $\mathcal{B}$ with $p = \sum (|c_i| - 2n)$.
\end{proposition}

We first show how to deduce Theorem~\ref{thm:main2} from
Proposition~\ref{prop:CobMain2}.
\begin{proof}[Proof of Theorem~\ref{thm:main2}]
  This is essentially the same proof as that of
  Theorem~\ref{thm:main2}.  With notation as in the proof of that
  theorem, the bordism class provided by
  proposition~\ref{prop:CobMain2} is in the image of some element of
  \begin{equation*}
    \Omega^{SO}_p(B\Diff(W \cup_P gK \cup_P M, M))
  \end{equation*}
  under the map in oriented bordism induced by the Pontryagin--Thom
  map
  \begin{equation*}
    B\Diff(W \cup_P gK \cup_P M, M) \lra \Omega^\infty_{(f)} \MT{SO}{2n}.
  \end{equation*}
  An element $\Omega^{SO}_p(B\Diff(W \cup_P gK \cup_P M, M))$ mapping
  to $[\pi,E,B,V,\phi]$, is given by a bundle $\pi: E \to B$ with
  fibre $W \cup_P gK \cup_P M$.
\end{proof}

Before proving Proposition~\ref{prop:CobMain2}, we first show it
suffices to establish the case when $b=0$ and $f=0$.

\begin{lemma}\label{lem:reduction}
  If Proposition \ref{prop:CobMain2} is true when $b=0$ and $f=0$ then
  it is true for arbitrary $b$ and $f$.
\end{lemma}
\begin{proof}
  Let $(f, e, b, K_C)$ be data we wish to realise. The associated data
  $(0, e - b \cdot f,0, K_C)$ still satisfies the compatibility
  conditions so by assumption we can realise it by some $[\pi', B',
  E', V', \varphi']$.

  Then we pick some $B'' \in b$ and set $B = B \amalg B''$.  Then $B
  \in b$, and the composition $\pi: E' \to B' \hookrightarrow B$
  realises the data $(0,e - b \cdot f, b, K_C)$.  Then we pick $F \in
  f$, set $E = E' \amalg (B \times F)$, and extend $\pi$ by the
  projection $B \times F \to B$.  Since the trivial bundle $B \times F
  \to B$ has vanishing characteristic numbers, this realises $(f, e,
  b, K_C)$.
\end{proof}

It remains to prove Proposition \ref{prop:CobMain2} in the case
$b=f=0$. When $f=0$ the space $\Omega^\infty_{(0)} \MT{SO}{d}$ still
consists of many components, and we shall in fact show something
stronger, that we may find an element in
$\Omega_p^{SO}(\Omega^\infty_0 \MT{SO}{d})$, the bordism of the
basepoint component. To satisfy the requirement $b=0$ we are asking
for an element of reduced bordism $\Omega_p^{SO}(\Omega^\infty_0
\MT{SO}{d},\ast)$.

Let us write $\MSO$ for the spectrum representing oriented bordism,
and $u : \MSO \to \mathbf{H\bZ}$ for the map of spectra representing
its Thom class. Let $\overline{\MSO}$ denote the homotopy fibre of
$u$, and $\overline{\Omega}_*^{SO}(-)$ be the associated homology
theory. Let us write $\mathbf{mtso}(d)$ for the $0$-connected cover of
$\MT{SO}{d}$. We will establish a commutative diagram
\begin{equation*}
  \xymatrix{
    \overline{\Omega}_p^{SO}(\Omega_0^\infty \MT{SO}{d},\ast) \otimes \bQ
    \ar@{->>}[r]^-{\sigma} \ar@{^{(}->}[d]&
    \overline{\Omega}_p^{SO}(\mathbf{mtso}(d)) \otimes \bQ
    \ar@{^{(}->}[d] \ar@{=}[r] &
    \overline{\Omega}_p^{SO}(\mathbf{mtso}(d)) \otimes \bQ \ar[d]\\
    \Omega_p^{SO}(\Omega_0^\infty \MT{SO}{d}, \ast) \otimes \bQ
    \ar@{->>}[r]^-{\sigma} \ar@{->>}[d]&
    \Omega_p^{SO}(\mathbf{mtso}(d)) \otimes \bQ
    \ar@{->>}[r]^-{\mathcal{E}} \ar@{->>}[d] & \Omega_{p+d}^{SO}
    \otimes \bQ \ar@{->>}[d] \\
    H_p(\Omega_0^\infty \MT{SO}{d}, \ast;\bQ)
    \ar@{->>}[r]^-{\sigma} & H_p(\mathbf{mtso}(d);\bQ) \ar@{->>}[r]
    \ar@{}[ur]|{\CircNum{1}}  & AP^{p+d}(d)^\vee
  }
\end{equation*}
in which all columns are exact at the second row, and maps are
injective or surjective as indicated.  Proposition \ref{prop:CobMain2}
for the case $b = f = 0$ concerns elements of
$\Omega_p^{SO}(\Omega_0^\infty \MT{SO}{d}, \ast) \otimes \bQ$ with
specified image in $H_p(\Omega_0^\infty \MT{SO}{d}, \ast;\bQ)$ and
$\Omega_{p+d}^{SO}$, and the proof follows by an easy diagram chase
once the diagram is established.

The maps $\sigma$ are realised by maps of spectra
\begin{equation*}
  \sigma : \Sigma^\infty \Omega_0^\infty \MT{SO}{d} \lra \mathbf{mtso}(d)
\end{equation*}
coming from the $\Sigma^\infty$--$\Omega^\infty$ adjunction, so the
left-hand half of the diagram commutes. The map $\sigma$ is surjective
on any rationalised homology theory.

The map $\mathcal{E}$ is realised by
\begin{equation*}
\MSO \wedge \mathbf{mtso}(d) \to \MSO \wedge \MT{SO}{d} \to \MSO \wedge \Sigma^{-d} \MSO \overset{\mu}\to \Sigma^{-d} \MSO
\end{equation*}
where the first map comes from $\mathbf{mtso}(d)$ being the connective
cover of $\MT{SO}{d}$, the second comes from the stabilisation and the
last map comes from the multiplicative structure of $\MSO$. On
rational cohomology we have a square
\begin{equation*}
  \xymatrix{
    {H^p(\MSO \wedge \mathbf{mtso}(d);\bQ)} & &
    {H^p(\Sigma^{-d}\MSO;\bQ)} \ar[ll]_-{\mathcal{E}^*}\\
    {H^p(\mathbf{mtso}(d);\bQ)} \ar[u] & {H^{p+d}(BSO(d);\bQ)}
    \ar[l]_{\text{Thom}}^\cong & {H^{p+d}(BSO;\bQ)}
    \ar[l]_-\rho\ar[u]^{\text{Thom}}_\cong
  }
\end{equation*}
\emph{which does not commute}. Unwinding definitions, we see that the
subset $AP^{p+d}(d)$ of $H^{p+d}(BSO;\bQ)$ is the largest subspace for
which the diagram does commute. Dualising we obtain the commutative
square $\CircNum{1}$.

\begin{lemma}
The map $AP^{*}(d) \hookrightarrow H^{*}(BSO;\bQ) \overset{\rho}\to H^{*}(BSO(d);\bQ)$ is injective for $d \geq 2$.
\end{lemma}
\begin{proof}
  An element $x$ in the kernel is simultaneously in $AP^{*}(d)$ and
  the ideal $(p_{\lfloor d/2 \rfloor+1}, p_{\lfloor d/2 \rfloor+2},
  \ldots) \subset H^*(BSO;\bQ)$. If $x$ is non-trivial, it must then
  have degree $\vert x \vert \geq \vert p_{\lfloor d/2 \rfloor+1}\vert
  = 4 \lfloor d/2 \rfloor+4$. However, by the classification of almost
  primitives in Proposition \ref{prop:Class}, a class which is almost
  primitive of order $d$ and of degree $> 2 \cdot d$ must be an
  ordinary primitive, so $x$ must be a multiple of a Pontryagin
  character class. But no Pontryagin character class lies in the ideal
  $(p_{\lfloor d/2 \rfloor+1}, p_{\lfloor d/2 \rfloor+2}, \ldots)$ if
  $d \geq 2$, as they contain a non-trivial monomial which is a power
  of $p_1$.
\end{proof}

It follows that the two horizontal maps in the square $\CircNum{1}$
are surjective for $d \geq 2$.  We now turn to the exactness of the
columns.  The first two follows for general reasons: a cofibration
sequence of spectra induces long exact sequences of the corresponding
homology theories.  It remains to consider the third column.

\begin{lemma}
  The sequence
  \begin{equation*}
    \overline{\Omega}_p^{SO}(\mathbf{mtso}(d)) \otimes \bQ \lra
    \Omega_{k+d}^{SO} \otimes \bQ \lra AP^{k+d}(d)^\vee \lra 0 
  \end{equation*}
  is exact.
\end{lemma}
\begin{proof}
  Dualising, we must show that if a class $x \in H^p(\Sigma^{-d}
  \MSO;\bQ)$ is zero under the composition
  \begin{equation*}
    \overline{\MSO} \wedge \mathbf{mtso}(d) \lra \MSO \wedge
    \Sigma^{-d} \MSO \lra \Sigma^{-d} \MSO, 
  \end{equation*}
  then considered as a class in $H^{k+d}(BSO;\bQ)$ it is almost
  primitive of order $d$. But on cohomology, after applying the Thom
  isomorphism, this composition is
  \begin{equation*}
    H^{*}(BSO) \overset{\Delta}\lra H^{*}(BSO) \otimes H^{*}(BSO)
    \overset{\text{proj}}\lra H^{* \geq 1}(BSO) \otimes H^{* \geq
      d+1}(BSO(d))
  \end{equation*}
  so if $x$ is sent to zero, it is \emph{by definition} almost
  primitive of order $d$.
\end{proof}

This finishes the proof of Proposition \ref{prop:CobMain2}.

\section{Proof of Proposition \ref{prop:Class}}

Recall that we write $\ph_i \in H^{4i}(BSO;\bQ)$ for the $i$th
Pontryagin character class, so that $H^*(BSO;\bQ) \cong \bQ[\ph_1,
\ph_2, \ldots]$.  We wish to show that the vector subspace $AP^*(d)
\subset \bQ[\ph_1, \ph_2, \ldots]$ is spanned by the monomials
$\ph_{a_1} \cdots \ph_{a_k}$ such that every proper submonomial
$\ph_{a_1} \cdots \ph_{a_{j-1}} \cdot \ph_{a_{j+1}} \cdots \ph_{a_k}$
has degree $\leq d$.  It is clear that all such monomials \emph{are}
almost primitive of degree $d$.  For a sequence $I = (i_1, i_2,
\dots)$ of non-negative integers with only finitely many non-zero
terms, we shall write
\begin{equation*}
  \ph^I = \prod_j \ph_j^{i_j}.
\end{equation*}
In this notation $\ph^I \cdot \ph^J = \ph^{I+J}$.  The $\ph^I \in
H^*(BSO;\bQ)$ form a basis, and we shall always express elements of
$H^*(BSO;\bQ)$ in this basis.

The following argument is similar to \cite[\S 6]{CCG}.  
\begin{claim}
  Every $\ph^J$ occurring as a proper factor of a monomial of $x \in
  AP^*(d)$ either has $\vert \ph^J \vert \leq d$ or $\vert \ph^J \vert
  > 4 \cdot \lfloor d/2 \rfloor$.
\end{claim}
\begin{proof}
  If $\ph^I \cdot \ph^J$ occurs as a monomial in $x$ with non-zero
  coefficient, we may write
  \begin{equation*}
    x = \ph^I\cdot \bigg(\sum_{J'} \alpha_{J'} \ph^{J'}\bigg) + y,
  \end{equation*}
  where $y$ is a linear combination of monomials not divisible by
  $\ph^I$, and the sum is over multiindices $J'$ with $|\ph^{J'}| =
  |\ph^J|$ (so $\alpha_J \neq 0$).  Then
  \begin{equation*}
    \Delta(x) =  \ph^I \otimes \bigg( \sum \alpha_{J'}
    \binom{I+J'}{I} \ph^{J'}\bigg) + z,
  \end{equation*}
  where we write $\binom{K}{I} = \prod_n\binom{k_n}{i_n}$ and $z$ is a
  linear combination of terms not of the form $\ph^I \otimes
  \ph^{J''}$.  As $x$ is almost primitive, we deduce that the element
  \begin{equation*}
    \sum \alpha_{J'}
    \binom{I+J'}{I} \ph^{J'} \in H^*(BSO;\bQ),
  \end{equation*}
  which is non-zero since $\alpha_J \binom{I+J}{I} \neq 0$, either has
  degree $\leq d$ or is in the kernel of the map $\rho: H^*(BSO;\bQ)
  \to H^*(BSO(d);\bQ)$.  Since this kernel vanishes in degrees $\leq
  4\lfloor d/2\rfloor$ we have proved the claim.
\end{proof}

\begin{proof}[Proof of Proposition~\ref{prop:Class}]
  To prove the proposition, suppose for contradiction that $x \in
  AP^*(d)$ contains a monomial of the form $\ph^I \cdot \ph^J$ with
  $|\ph^I| > 0$ and $|\ph^J| > d$.  We may assume that $|\ph^J|$ is
  minimal with this property.  The claim above shows that in fact
  $|\ph^J|/4 > \lfloor d/2\rfloor$ so in particular $|\ph^J| > 2d$.  It
  follows that $\ph^J = \ph_j$ for some $j$, since any proper factor
  of $\ph^J$ must have degree $\leq d$ by minimality of $J$ and hence
  can be combined to a proper factor of degree $\in (d,2d]$
  contradicting the statement of the claim.

  We have proved that $x$ contains a monomial of the form $\ph^I\cdot
  \ph_j$ with $|\ph_j| > 2d$.  As in the proof of the claim, we may
  write
  \begin{equation*}
    x = \ph_j\cdot \bigg(\sum_{I'} \alpha_{I'} \ph^{I'}\bigg) + y,
  \end{equation*}
  where $y$ is a sum of monomials not divisible by $\ph_j$, and
  $\alpha_I \neq 0$.  Then
  \begin{equation*}
    \Delta(x) = \bigg(\sum_{I'} \alpha_{I'} (1 +
    i'_j)\ph^{I'}\bigg) \otimes \ph_j + z
  \end{equation*}
  where $z$ is a linear combination of terms not of the form $\ph^{I'}
  \otimes \ph_j$.  As $|\ph_j| > d$ and $x \in AP^*(d)$, it follows
  that $\ph_j$ is in the kernel of $H^*(BSO;\bQ) \to H^*(BSO(d);\bQ)$.
  But this is a contradiction since $d \geq 2$, and $\ph_j$ is
  non-zero even in $H^*(BSO(2);\bQ)$.
\end{proof}

\begin{remark}
  The \emph{near-primitive elements of order $d$} discussed by Church,
  Crossley and Giansiracusa are the almost primitive elements of order
  $d$ for which in addition $\sum_j x^p_j \otimes x^{d}_j=0$. We write
  $NP^*(d) \subset AP^*(d)$ for that subspace. In \cite[Theorem
  A]{CCG} they give an explicit description of the near primitive
  elements, which combined with Proposition \ref{prop:Class} shows
  that $AP^*(d) = NP^*(d+1)$.
\end{remark}
